\documentclass[11pt]{article}
\usepackage{amsmath,amsthm,amssymb}
\usepackage[hyperfootnotes=false]{hyperref}
\usepackage{color}
\usepackage{cancel}
\usepackage{titlesec}
\usepackage{enumerate}
\usepackage[normalem]{ulem}
\titleformat{\paragraph}
{\normalfont\normalsize\bfseries}{\theparagraph}{1em}{}
\titlespacing*{\paragraph}
{0pt}{3.25ex plus 1ex minus .2ex}{1.5ex plus .2ex}

\def\titlerunning#1{\gdef\titrun{#1}}
\makeatletter
\def\author#1{\gdef\autrun{\def\and{\unskip, }#1}\gdef\@author{#1}}
\def\address#1{{\def\and{\\\hspace*{18pt}}\renewcommand{\thefootnote}{}%
\footnote {#1}}%
\markboth{\autrun}{\titrun}}
\makeatother
\def\email#1{\hspace*{4pt}{\em e-mail}: #1}
\def\MSC#1{{\renewcommand{\thefootnote}{}%
\footnote{\emph{Mathematics Subject Classification (2020):} #1}}}
\def\keywords#1{\par\medskip
\noindent\textbf{Keywords:} #1}

\newtheorem{theorem}{Theorem}[section]

\newtheorem{prop}[theorem]{Proposition}
\newtheorem{cor}[theorem]{Corollary}
\newtheorem{lemma}[theorem]{Lemma}

\newtheorem{defin}[theorem]{Definition}

\theoremstyle{definition}

\newtheorem{remark}[theorem]{Remark}

\numberwithin{equation}{section}

\DeclareMathOperator{\N}{N}

\def\cC{\mathcal C}

\def\cS{\mathcal S}

\def\cV{\mathcal V}

\def\PG{{\rm PG}}

\def\F{{\mathbb F}}

\newcommand{\ep}{\epsilon}

\def\rk{{\rm rk}}

\def\w{\omega}

\begin{document}


\baselineskip=16pt

\titlerunning{}

\title{Small complete caps in $\PG(4n+1, q)$}

\author{Antonio Cossidente
\and 
Bence Csajb\'ok
\and
Giuseppe Marino 
\and 
Francesco Pavese}

\date{}

\maketitle

\address{A. Cossidente: Dipartimento di Matematica, Informatica ed Economia, Universit{\`a} degli Studi della Basilicata, Contrada Macchia Romana, 85100, Potenza, Italy; \email{antonio.cossidente@unibas.it}
\and 
B. Csajb\'ok: ELKH--ELTE Geometric and Algebraic Combinatorics Research Group, ELTE E\"otv\"os Lor\'and University, Budapest, Hungary, Department of Geometry,
	1117 Budapest, P\'azm\'any P.\ stny.\ 1/C, Hungary;	\email{bence.csajbok@ttk.elte.hu}
\and
G. Marino: Dipartimento di Matematica e Applicazioni ``Renato Caccioppoli'', Universit{\`a} degli Studi di Napoli ``Federico II'', Complesso Universitario di Monte Sant'Angelo, Cupa Nuova Cintia 21, 80126, Napoli, Italy; \email{giuseppe.marino@unina.it}
\and
F. Pavese: Dipartimento di Meccanica, Matematica e Management, Politecnico di Bari, Via Orabona 4, 70125 Bari, Italy; \email{francesco.pavese@poliba.it}
}


\MSC{Primary: 51E22 Secondary: 11T06; 51E20; 94B05.}

\begin{abstract}
In this paper we prove the existence of a complete cap of $\PG(4n+1, q)$ of size 
$2(q^{2n+1}-1)/(q-1)$, for each prime power $q>2$.
It is obtained by projecting two disjoint Veronese varieties of $\PG(2n^2+3n, q)$ from a suitable $(2n^2-n-2)$-dimensional projective space. This shows that the trivial lower bound for the size of the smallest complete cap of $\PG(4n+1, q)$ is essentially sharp. 

\end{abstract}

\section{Introduction}

Let $\PG(r, q)$ denote the $r$-dimensional projective space over $\F_q$, the finite field with $q$ elements. A {\em $k$-cap} in $\PG(r, q)$ is a set of $k$ points no three of which are collinear. A $k$-cap in $\PG(r, q)$ is said to be {\em complete} if it is not contained in a $(k+1)$-cap in $\PG(r, q)$. The study of caps is not only of geometrical interest, their concept arises from coding theory. 

A {\em $q$--ary linear code} $\cC$ of dimension $k$ and length $N$ is a $k$-dimensional vector subspace of $\F_q^N$, whose elements are called {\em codewords}. A {\em generator matrix} of $\cC$ is a matrix whose rows form a basis of $\cC$ as an $\F_q$-vector space. The {\em minimum distance of $\cC$} is $d = \min\{d(u, 0) \mid u \in \cC, u \ne 0\}$, where $d(u, v)$, $u, v \in \F_q^N$, is the {\em Hamming distance} on $\F_q^N$. A vector $u$ is {\em $\rho$--covered by $v$} if $d(u, v) \le \rho$. The {\em covering radius} of a code $\cC$ is the smallest integer $\rho$ such that every vector of $\F_q^n$ is $\rho$--covered by at least one codeword of $\cC$. 
A linear code with minimum distance $d$ and covering radius $\rho$ is said to be an $[N, k, d]_q$ $\rho$--code. 
For a code $\cC$, its {\em dual code} is $\cC^\perp = \{v \in \F_q^N \mid v \cdot c = 0, \forall c \in \cC\}$ (here $\cdot$ is the Euclidean inner product). The dimension of the dual code $\cC^\perp$, or the codimension of $\cC$, is $N - k$. Any matrix which is a generator matrix of $\cC^\perp$ is called a {\em parity check matrix} of $\cC$. If $\cC$ is linear 
with parity check matrix $M$, its covering radius is the smallest $\rho$ such that every $w \in \F_q^{N - k}$ can be written as a linear combination of at most $\rho$ columns of $M$. 
By identifying the representatives of the points of a complete $k$-cap of $\PG(r, q)$ with columns of a parity check matrix of a $q$--ary linear code it follows that (apart from three sporadic exceptions) complete $k$-caps in $\PG(r, q)$ with $k > r + 1$ and non-extendable linear $[k, k - r - 1, 4]_q$ $2$--codes are equivalent objects, see \cite{GDT}.

One of the main issues in this area is to determine the spectrum of the sizes of complete caps in a given projective space and in particular 
to determine the size of the smallest and the largest complete caps. 
For the size $t_2(r, q)$ of the smallest complete cap in $\PG(r, q)$, the trivial lower bound is $t_2(r, q) > \sqrt{2} q^{\frac{r-1}{2}}$. Constructions of complete caps whose size is close to this lower bound are only known for $q$ even. In $1959$ B. Segre \cite{Segre} proved the existence of a complete cap of $\PG(3, q)$, $q$ even, of size $3q+2$. Later on, F. Pambianco and L. Storme \cite{PS}, 
based on Segre's construction, 
proved the existence of complete caps of $\PG(r, q)$, $q$ even, of size $q^{\frac{r}{2}} + 3 \left( q^{\frac{r-2}{2}} + \dots + q \right) + 2$, if $r$ is even, and $3 \left( q^{\frac{r-1}{2}} + \dots + q \right) + 2$, if $r$ is odd. These ideas have been further developed by M. Giulietti \cite{G} and A. Davydov, M. Giulietti, S. Marcugini and F. Pambianco \cite{DGMP} who exhibited smaller caps in $\PG(r, q)$, $q$ even. A probabilistic approach has been employed in \cite{DFMP} to provide an upper bound on $t_2(r, q)$. For further or more recent results on this topic see also \cite{BFMP, BFMP1, BGMP, DO}. For an account on the various constructive methods known so far the reader is referred to \cite{G1}. Hence despite the many efforts made, apart from the cases $q$ even and $r$ odd, all known  explicit constructions of infinite families of complete caps in $\PG(r, q)$ have size far from the trivial bound.

In this paper we show that the trivial lower bound on $t_2(4n+1, q)$ is essentially sharp. More precisely, we prove the existence of a complete cap of $\PG(4n+1, q)$, $q>2$, of size $2(q^{2n} + \dots + 1)$ that is obtained by projecting two disjoint Veronese varieties of $\PG(2n^2+3n, q)$ from a suitable $(2n^2-n-2)$-dimensional projective space. To prove the completeness of these caps we also rely on a result of F. \"Ozbudak about linearized permutations \cite{Ferruh}. We use his result to show that certain linearized polynomials defined over $\F_{q^k}$ induce a permutation of $\F_{q^k}$ if and only if they induce a permutation of $\F_{q^{2k}}$. 

\section{Preliminaries}
We start with some technical, preliminary results that we will use in the next section.

\noindent Let $\N_{q^k/q}: x \in \F_{q^{k}} \mapsto x^{q^{k-1} + \dots + q + 1} \in \F_q$ be the {\em norm function} of $\F_{q^{k}}$ over $\F_q$.

\begin{lemma}
		\label{l1}
		Take $a\in \F_{q^k}$ and $b\in \F_{q^{2k}}$ such that $b^{q^k}=-b$ and $a^2-b^2=1$. 
		Then $\N_{q^{2k}/q}(a+b)=\N_{q^{2k}/q}(a-b)=1$.	
\end{lemma}
\begin{proof}
Both of $\N_{q^{2k}/q}(a+b)$ and $\N_{q^{2k}/q}(a-b)$ are equal to
	\[(a+b)(a-b)(a^q+b^q)(a^q-a^q)\ldots(a^{q^k}+b^{q^k})(a^{q^k}-b^{q^k})=\N_{q^k/q}(a^2-b^2)=1.\]
\end{proof}
	
	\begin{cor}
		\label{cor1}
		If $\w\in \F_{q^k}$ and $\w^2-1$ is a non-square in $\F_{q^k}$, then \[\N_{q^{2k}/q}\left(\w+T\right) = \N_{q^{2k}/q}\left(\w-T\right) = 1,\] where $T\in \F_{q^{2k}}$ is any of the square-roots of $\w^2-1$.
	\end{cor}
\begin{proof}
	In Lemma \ref{l1} put $a=\w$ and $b=T$. The condition \[T^{q^k}=T(\w^2-1)^{(q^k-1)/2}=-T\] holds true and hence the assertion follows from Lemma \ref{l1}.
\end{proof}

	\begin{lemma}
		\label{sum}
		For $\epsilon \in \{1,-1\}$ consider the following $k\times k$ matrix:
		\[
		D_{\epsilon}:=\begin{pmatrix}
		a_0 & a_1 & a_2 & 0 & \ldots & 0 & 0 \\
		0 & a_0^q & a_1^q & a_2^q & \ldots & 0 & 0 \\
		\vdots & 	\vdots & 	\vdots & 	\vdots & 	\vdots & 	\vdots & 	\vdots \\
		\epsilon a_2^{q^{k-2}} & 0 & 0 & 0 & \ldots & a_0^{q^{k-2}} & a_1^{q^{k-2}} \\
		\epsilon a_1^{q^{k-1}} & \epsilon a_2^{q^{k-1}} & 0 & 0 & \ldots & 0 & a_0^{q^{k-1}} 
		\end{pmatrix}.\]
		Then $\det D_1 + \det D_{-1}=2\left(\N_{q^k/q}(a_0)+\N_{q^k/q}(a_2)\right)$.
	\end{lemma}
	\begin{proof}
		For $i\neq j$ let $D_{\epsilon,i,j}$ denote the $(k-2)\times (k-2)$ matrix obtained from $D_{\epsilon}$ after removing its first two columns and its $i$-th and $j$-th rows. Then 	
		\begin{align*}
		    \det D_{\epsilon} = & a_0^{q+1}\det D_{\epsilon,1,2}+(-1)^{k} a_0\epsilon a_2^{q^{k-1}} \det D_{\epsilon,1,k}+\ep^2 a_2^{q^{k-2}+q^{k-1}}\det D_{\ep, k-1,k}+ \\
			     & (-1)^{k+1}\ep a_2^{q^{k-2}}a_1\det D_{\ep, 1,k-1} + (-1)^{k} \ep a_2^{q^{k-2}}a_0^q \det D_{\ep,2,k-1} + \\
	        &	(-1)^{k}\ep a_1^{q^{k-1}+1}\det D_{\ep, 1, k}+(-1)^{k+1}\ep a_1^{q^{k-1}}a_0^q\det D_{\ep, 2, k}.
		\end{align*}
		Note that $D_{\ep,k-1,k}$ is upper-triangular, thus $\det D_{\ep,k-1,k}=a_2^{1+q+\ldots+q^{k-3}}$. Also $D_{\ep,1,2}$ is upper-triangular, thus
		$\det D_{\ep,1,2}=a_0^{q^2+q^3+\ldots+q^{k-1}}$. Then the result follows since $\ep^2=1$ and the minors $\det D_{\epsilon,i,j}$ do not depend on $\epsilon$.
	\end{proof}
	
	\begin{theorem}
		\label{fer}
		Consider $f(x)=a_0x+a_1x^q+a_2x^{q^2}\in \F_{q^k}[x]$. If $\N_{q^k/q}(a_0)+\N_{q^k/q}(a_2)=0$ then 
		$f(x)$ induces a permutation of $\F_{q^k}$ if and only if it induces a permutation of $\F_{q^{2k}}$.
	\end{theorem}
	\begin{proof}
		By \cite[Theorem 3.6]{Ferruh} $f$ induces a permutation of $\F_{q^{2k}}$ if and only if $\det D_1 \neq 0$ and $\det D_{-1}\neq 0$, where $D_1$ and $D_{-1}$ are defined as in Lemma \ref{sum}. The polynomial $f(x)$ induces a permutation of $\F_{q^k}$ if and only if $\det D_1 \neq 0$. 
If $\N_{q^k/q}(a_0)+\N_{q^k/q}(a_2)=0$ then Lemma \ref{sum} gives $\det D_1 = - \det D_{-1}$, and hence	$\det D_1 \neq 0$ if and only if $\det D_{-1} \neq 0$. This finishes the proof.
	\end{proof}

\begin{prop}
\label{gcd}
$\gcd(q^{2n+1}-1,q+1)=1$ for $q$ even and $\gcd(q^{2n+1}-1,q+1)=2$ for $q$ odd.
\end{prop}
\begin{proof}
Let $p$ be a prime dividing both $q+1$ and $q^{2n+1}-1$. Hence $q+1 = pk$, for some integer $k$ and $q^{2n+1} - 1 = -2 + pkh$, for some integer $h$. Thus $p$ has to divide $2$, and hence $p=2$ if $q$ is odd and there is no such $p$ if $q$ is even.
\end{proof}

\section{The construction}\label{construction}

In $\F_{q^{2n+1}}^{4n+2}$ consider the $(4n+2)$-dimensional $\F_q$-subspace $V$ given by the set of vectors
\[ \left\{ v(a,b):=\left(a, b, a^q, b^q, \dots, a^{q^{2n}}, b^{q^{2n}}\right) : a, b \in \F_{q^{2n+1}}\right\}.\]
Then $\PG(V)$ is a $(4n+1)$-dimensional projective space over $\F_q$, since $\dim(V) = 4n+2$. 
For $(a,b)\neq (0,0)$ denote by $P(a,b)$ the point of $\PG(V)$ defined by the vector $v(a,b)$. 
Let $\Pi_1$ (resp. $\Pi_2$) be the $2n$-dimensional projective subspace consisting of the points $P(0,b)$ (resp. $P(a,0)$). For $c,d \in \F_q$ we will denote by 
$c P(a,b)+dP(a',b')$ the point of $\PG(V)$ on the line joining $P(a,b)$ and $P(a',b')$, defined by the vector $v(ca+da',cb+db')$. Let us define the following pointsets of $\PG(V)$:
\[ \cV_{\w} := \left\{ P(x^2,\omega x^{q+1}) : x \in \F_{q^{2n+1}} \setminus \{0\}\right\}, \]
where $\w \in \F_{q^{2n+1}} \setminus \{0\}$. Consider the invertible linear map of $V$ given by
\begin{align}
& v(a, b) \in V \mapsto v(\eta^2 a, \eta^{q+1} b) \in V, \label{phi}
\end{align}
where $\eta$ is a primitive element of $\F_{q^{2n+1}}$ and denote by $\phi$ the projectivity of $\PG(V)$ induced by \eqref{phi}. 

\begin{lemma}
	\label{lemma1}
\begin{enumerate}[\rm(i)]
    \item $|\cV_\w| = \frac{q^{2n+1}-1}{q-1}$,
    \item $\langle \phi \rangle$ is a group of order $\frac{q^{2n+1}-1}{q-1}$ acting regularly on points of $\Pi_1$, $\Pi_2$ and $\cV_\w$, $\w \in \F_{q^{2n+1}} \setminus \{0\}$.
\item If $\delta, \w \in \F_{q^{2n+1}}\setminus \{0\}$ then the projectivity induced by $v(a,b)\mapsto v(a, \delta  b)$ maps $\cV_{\w}$ to $\cV_{ \delta \w}$ and fixes $\Pi_1$ and $\Pi_2$.
\end{enumerate}
\end{lemma}
\begin{proof}
$(i)$ Let $P(x^2, \w x^{q+1})$, $P(y^2, \w y^{q+1})$ be points of $\cV_\w$. Thus $P(x^2, \w x^{q+1})=P(y^2, \w y^{q+1})$ if and only if $x/y \in \F_q$. 

$(ii)$ Consider the projectivity $\phi^i$ associated with the linear map given by $v(a, b) \in V \mapsto v(\eta^{2i} a, \eta^{(q+1)i} b) \in V$. Then $\phi^i$ is the identity if and only if $\eta^{(q-1)i} = 1$ and hence $\eta^i \in \F_q$. Moreover $|Stab_{\langle \phi \rangle}(P(1, 0))| = |Stab_{\langle \phi \rangle}(P(0, 1))| = |Stab_{\langle \phi \rangle}(P(1, \w))| = 1$. Since $|\langle \phi \rangle| = |\cV_{\w}| = |\Pi_1| = |\Pi_2|$, the result follows. 

$(iii)$ Straightforward.
\end{proof}

\begin{lemma}
	\label{lemma2}
The pointset of $\PG(V) \setminus \left( \Pi_1 \cup \Pi_2 \right)$ is partitioned into $\bigcup_{\w \in \F_{q^{2n+1}} \setminus \{0\}} \cV_w$. 
\end{lemma}
\begin{proof}
If $|\cV_{\w_1} \cap \cV_{\w_2}| \ne 0$, there exist $x, y \in \F_{q^{2n+1}} \setminus \{0\}$ such that $\frac{x^2}{y^2} = \frac{\w_1 x^{q+1}}{\w_2 y^{q+1}} \in \F_q$. Then $\left(\frac{x}{y}\right)^{2(q-1)} = 1$ and $\left(\frac{x}{y}\right)^{q-1} = \pm 1$. Note that necessarily $\left(\frac{x}{y}\right)^{q-1} = 1$, otherwise $1 = \N_{q^{2n+1}/q}\left(\left(\frac{x}{y}\right)^{q-1}\right) = \N_{q^{2n+1}/q}(-1) = -1$. Hence we have that $\w_1 \left(\frac{x}{y}\right)^{q-1} = \w_2$, that is, $\w_1 = \w_2$.  
\end{proof}
\begin{prop} \label{cap}
The set $\cV_\w$ is a cap of $\PG(V)$.
\end{prop}
\begin{proof}
Assume by contradiction that there are three points $P_1, P_2, P_3$ of $\cV_\w$ that are collinear. By Lemma \ref{lemma1}, we may assume that $\w=1$ and $P_1 = P(1,1)$. Let $P_2 = P(x^2, x^{q+1})$, $P_3 = P(y^2, y^{q+1})$, where $x, y \in \F_{q^{2n+1}} \setminus \F_q$ and $x/y \notin \F_q$. Hence there exist $\lambda, \rho \in \F_q \setminus \{0\}$ such that 
\begin{align*}
    & x^2 + \lambda y^2 = \rho, \\
    & x^{q+1} + \lambda y^{q+1} = \rho .
\end{align*}
Therefore $x^{2(q+1)} = (\rho - \lambda y^2)^{q+1} = (\rho - \lambda y^{q+1})^2$, that is, $(y - y^q)^2 = 0$, a contradiction.
\end{proof}

\begin{defin}
	Let $\alpha$ be $-1$ if $q$ is odd and an element of $\F_q \setminus \{0, 1\}$ if $q>2$ is even.
\end{defin}

\underline{From now on, we will assume $q>2$.}

\begin{prop} \label{cap1}
The set $\cV_1 \cup \cV_\alpha$ is a cap of $\PG(V)$. 
\end{prop}
\begin{proof}
Assume by contradiction that there are three points $P_1, P_2, P_3$ of $\cV_1 \cup \cV_{\alpha}$ that are collinear. 

Let $q$ be odd. We may assume that $P_1 \in \cV_{-1}$ and $P_2, P_3 \in \cV_1$. Indeed, if two of these three points were on $\cV_{-1}$ we can apply the involutory projectivity of $\PG(V)$ induced by $v(a, b) \in V \mapsto v(a, -b) \in V$ switching $\cV_1$ and $\cV_{-1}$. By Lemma \ref{lemma1}, we may assume that $P_1 = P(1,-1)$. Let $P_2 = P(x^2, x^{q+1})$, $P_3 = P(y^2, y^{q+1})$, where $x/y \notin \F_q$. Hence there exist $\lambda, \rho \in \F_q \setminus \{0\}$ such that 
\begin{align*}
    & x^2 + \lambda y^2 = \rho, \\
    & x^{q+1} + \lambda y^{q+1} = - \rho.
\end{align*}
Therefore $x^{2(q+1)} = (\rho - \lambda y^2)^{q+1} = (\rho + \lambda y^{q+1})^2$, that is, $(y + y^q)^2 = 0$. In particular $1 = \N_{q^{2n+1}/q}(y^{q-1}) = \N_{q^{2n+1}/q}(-1) = -1$, a contradiction.

Let $q$ be even. 
By Lemma \ref{lemma1} $(iii)$ the projectivity $\pi$ induced by the linear map $v(a,b)\mapsto v(a,b \alpha^{-1})$ maps $\cV_{\alpha}$ to $\cV_1$ and $\cV_1$ to $\cV_{\alpha^{-1}}$. 
We may choose indices for the three collinear points such that either $P_1 \in \cV_\alpha$ and $P_2,P_3 \in \cV_1$, or 
$P_1 \in \cV_1$ and $P_2,P_3 \in \cV_\alpha$, i.e. 
$P_1^\pi \in \cV_{\alpha^{-1}}$ and $P_2^\pi,P_3^\pi \in \cV_1$.
We provide the proof for the former case. The latter case follows in the same way with $\alpha^{-1} \in \F_q \setminus \{0,1\}$ playing the role of $\alpha$ and $P_i^\pi$ playing the role of $P_i$, for $i=1,2,3$.

By Lemma \ref{lemma1} 
we may assume also $P_1=P(1,\alpha) \in \cV_{\alpha}$ and put $P_2 = P(x^2, x^{q+1})$, $P_3 = (y^2, y^{q+1})$, where  $x/y \notin \F_q$. Hence there exist $\lambda, \rho \in \F_q \setminus \{0\}$ such that 
\begin{align*}
    & x^2 + \lambda y^2 = \rho, \\
    & x^{q+1} + \lambda y^{q+1} = \alpha \rho .
\end{align*}
Therefore $x^{2(q+1)} = (\rho + \lambda y^2)^{q+1} = (\alpha \rho + \lambda y^{q+1})^2$, that is, $y + y^q = \sqrt{\frac{\rho (1+\alpha^2)}{\lambda}} \in \F_q \setminus \{0\}$. In particular $y^{q^2} = y$, i.e., $y \in \F_{q^2}\cap \F_{q^{2n+1}}=\F_q$, a contradiction, since then $y+y^q$ was $0$. 
\end{proof}

\begin{theorem}\label{main}
Every point of $\PG(V) \setminus (\cV_1 \cup \cV_\alpha)$ lies on at least a line joining a point of $\cV_1$ and a point of $\cV_\alpha$. 
\end{theorem}
\begin{proof}
Let $R$ be a point of $\PG(V) \setminus (\cV_1 \cup \cV_\alpha)$. If $R$ belongs to $\Pi_1$ or $\Pi_2$, we may assume by Lemma \ref{lemma1} that $R =P (0,1)$ or $R =P (1,0)$, respectively. If $R \notin \Pi_1 \cup \Pi_2$, by Lemma \ref{lemma2}, there is $\w \in \F_{q^{2n+1}} \setminus \{0, 1, \alpha\}$ such that $R \in \cV_\w$. Moreover, by Lemma \ref{lemma1}, we may assume that $R = P(1, \w)$.   

\underline{Let $q$ be even.} If $R = P(0,1) \in \Pi_1$, then $R=P(1,1) + P(1, \alpha)$. Similarly, if $R = P(1,0) \in \Pi_2$, then $R = \alpha P(1,1) + P(1, \alpha)$. Assume that $R=P(1,\w) \notin \Pi_1 \cup \Pi_2$. Let $t \in \F_{q^{2n+1}}$ be such that 
\begin{equation}
\label{t}
t^{q+1} = \frac{\w(\alpha+1) + \alpha}{\alpha^2 + 1}.
\end{equation}
Note that such a $t$ always exists because of Proposition \ref{gcd}. Define $x = t + \frac{\alpha}{\alpha+1}$ and $y = t + \frac{1}{\alpha + 1}$. From \eqref{t}, since $\w \in \F_{q^{2n+1}} \setminus \{0,1,\alpha\}$ and $\alpha \in \F_q \setminus \{0,1\}$, it follows that $x\neq 0$ and $y\neq 0$. Then straightforward calculations show that the following equations are satisfied
\begin{align*}
    & x^2 + y^2 = 1, \\
    & x^{q+1} + \alpha y^{q+1} = \w .
\end{align*}
Therefore $R = P(x^2, x^{q+1})+ P(y^2, \alpha y^{q+1})$.

\underline{Let $q$ be odd.} If $R = P(0,1) \in \Pi_1$, then $R = P(1,1) - P(1, -1)$. Similarly, if $R = P(1,0) \in \Pi_2$, then $R = P(1,1) + P(1, -1)$. Assume that $R = P(1, \w) \notin \Pi_1 \cup \Pi_2$, with $\w \in \F_{q^{2n+1}} \setminus \{0, 1, -1\}$. We want to show that for all $\w \in \F_{q^{2n+1}} \setminus \{0,1,-1\}$, there exist $\lambda, \rho \in \F_q \setminus \{0\}$ and $x, y \in \F_{q^{2n+1}} \setminus \{0\}$ such that the following equations are satisfied
\begin{align}
& x^2 + \lambda y^2 = \rho, \label{eq1} \\
& x^{q+1} - \lambda y^{q+1} = \rho \w. \label{eq2} 
\end{align} 
If we can guarantee this, then $R=P(x^2,x^{q+1})+\lambda P(y^2,-y^{q+1})$. 


Clearly $(x^2)^{q+1}=(x^{q+1})^2$ and hence necessarily $(\rho - \lambda y^2)^{q+1} - (\rho \w + \lambda y^{q+1})^2 = 0$. Then $\rho-\lambda y^2 = \lambda y^{2q}+\rho \w^2 + 2\w \lambda y^{q+1}$ and after subtracting $\w^2(\rho-\lambda y^2)$ from both sides we have:
\begin{align*}
    & (\rho-\lambda y^2)-\w^2(\rho-\lambda y^2)=(\lambda y^{2q}+\rho \w^2 + 2\w \lambda y^{q+1})-\w^2(\rho-\lambda y^2), \\
    & (\rho-\lambda y^2)(1-\w^2)=\lambda(\w y+y^q)^2.
\end{align*}
By \eqref{eq1}, we get $x^2 = -\lambda (\w y + y^q)^2/(\w^2-1)$. Hence $-\lambda/(\w^2-1)$ has to be a non-zero square in $\F_{q^{2n+1}}$. This can be obtained for example by taking 
\begin{equation}
\label{lambda}
\lambda=-\N_{q^{2n+1}/q}(\w^2-1),
\end{equation}
which we will assume from now. Then
\begin{align}
&    x = \pm (\w^2-1)^{\frac{q+q^2+\ldots+q^{2n}}{2}} \left(\w y + y^q\right). \label{eq3}
\end{align}
Plugging \eqref{eq3} in \eqref{eq1} we obtain:
\begin{equation}
\label{rho}
\rho=-\lambda \left(\frac{y^2+y^{2q}+2\w y^{q+1}}{\w^2-1}\right).
\end{equation}
Since $\rho$ has to be in $\F_q \setminus \{0\}$, we must have $\rho^q=\rho$, which holds if and only if
\begin{align*}
    & (\w^2-1) (y^{2q^2} + 2\w^q y^{q^2+q}) + (\w^2-\w^{2q}) y^{2q} - (\w^{2q}-1) (2\w y^{q+1} + y^2) \\
    & = (\w^2-1) F_1(y) F_2(y) = 0,
\end{align*}
where
\begin{align*}
    & F_1(y) = y^{q^2} + \left(\w^q + \w (\w^2-1)^{\frac{q-1}{2}}\right) y^q + (\w^2-1)^{\frac{q-1}{2}} y, \\
    & F_2(y) = y^{q^2} + \left(\w^q - \w (\w^2-1)^{\frac{q-1}{2}}\right) y^q - (\w^2-1)^{\frac{q-1}{2}} y. 
\end{align*}
If $(\w^2-1)$ is a square in $\F_{q^{2n+1}}$, then
\[x^{q+1}=\frac{-\lambda}{(\w^2-1)^{\frac{q+1}{2}}}(\w y+y^q)^{q+1},\]
if $(\w^2-1)$ is a non-square in $\F_{q^{2n+1}}$, then
\[x^{q+1}=\frac{\lambda}{(\w^2-1)^{\frac{q+1}{2}}}(\w y+y^q)^{q+1}.\]
In the former case \eqref{eq2} is equivalent to 
\[\frac{\w y+y^q}{(\w^2-1)^{\frac{q+1}{2}}} F_2(y)=0,\]
in the latter case \eqref{eq2} is equivalent to 
\[\frac{\w y+y^q}{(\w^2-1)^{\frac{q+1}{2}}} F_1(y)=0.\]
Therefore, equations \eqref{eq1} and \eqref{eq2} are both satisfied  with $y\in \F_{q^{2n+1}}\setminus \{0\}$, where $y$ is a root of $F_1$ or $F_2$ according as $\w^2-1$ is a non-square or a square in $\F_{q^{2n+1}}$, respectively, and with $\lambda$, $x$, $\rho$ defined as in  \eqref{lambda},\eqref{eq3},\eqref{rho}. It is easy to see that $\w y + y^q = 0$ and $F_i(y)=0$ for some $i\in\{1,2\}$ cannot hold simultaneously and hence $x\neq 0$. 
Moreover $\rho=0$ if and only if $y^2+y^{2q}+2\w y^{q+1}=0$, or equivalently $(y\w+y^q)^2=y^2(\w^2-1)$. This clearly cannot hold when $\w^2-1$ is a non-square in $\F_{q^{2n+1}}$. The case when $\w^2-1$ is a square will be examined later.

Denote by $T\in \F_{q^{4n+2}}$ one of the square-roots of $\w^2-1$ and note that $(\w-T)(\w+T)=1$.

\medskip

CASE 1: Assume that $\w^2-1$ is a square in $\F_{q^{2n+1}}$. Choose some $\delta \in \F_q$ such that $\delta \left(\w - T\right)$ is a non-zero square in $\F_{q^{2n+1}}$. 
Note that for $T$ we have two choices, but the choice of $\delta$ does not depend on which square-root of $\w^2-1$ we took since $(\w-T)(\w+T)$ is a square and hence $(\w-T)$ is a square (non-square) if and only if $(\w+T)$ is a square (non-square). One might choose $\delta$ to be
$\N_{q^{2n+1}/q}(\w-T)$ since then $\delta(\w-T)=(\w-T)^{2+q+\ldots+q^{2n}}$, which is indeed a square, however, in what follows we won't need the exact value of $\delta$. 
Let $u \in \F_{q^{2n+1}}$ be such that $u^{q+1} = \delta \left(\w - T\right)$. Such an element $u$ always exists because of Proposition \ref{gcd}. Define $y$ as
\begin{equation}
\label{square_y}
y = \left(\w + T\right) u^q - u . 
\end{equation}
Then $y\neq 0$ since otherwise $u=u^q(\w+T)$ and hence
$u^{q+1}=\delta(\w-T)$ would give the equations
\begin{align*}
    & u^{2q}(\w+T)=\delta(\w-T), \\
    & u^2/(\w+T)=\delta(\w-T).
\end{align*}
The latter one is equivalent to $u^2=\delta$, and substituting this to the former one gives $\w+T=\w-T$, a contradiction.
Some calculations show that the following equalities hold true:
\begin{align*}
    y^q & = \left(\w + T\right) u - u^q, \\
    y^{q^2} & = \left(\w^q + T^{q}\right) u^q - \left(\w^q - T^{q}\right) \left(\w + T\right) u, \\
    F_2(y) & = u T^{q-1} \left(T^2 - \w^2+1\right) = 0.
\end{align*}
It remains to prove that $\rho$ defined as in \eqref{rho} is non-zero, that is, $y^2+y^{2q}+2\w y^{q+1}=(y\w+y^q)^2-y^2(\w^2-1) \neq 0$. The contrary holds if and only if
$y\w + y^q - y T =0$ or $y\w + y^q + y T =0$. 
In the former case, by \eqref{square_y}:
\begin{align*}
    y (\w-T) + y^q & = ((\w+T)u^q-u)(\w-T) +(\w+T)u-u^q \\
    & = 2Tu-u^q(T^2+1-\w^2) \\
    & = 0,
\end{align*}
a contradiction since 
$T^2+1-\w^2=0$ and $2Tu \neq 0$.
In the latter case, by \eqref{square_y}:
\begin{align*}
    y (\w+T) + y^q & = ((\w+T)u^q-u)(\w+T)+(\w+T)u-u^q \\
    & = u^q ((\w + T)^2-1)\\
    & = 0,
\end{align*}
that is, $T^2+2\w T+ \w^2-1=2T^2+2\w T= 0$, and hence $\w + T = 0$, a contradiction.
\medskip

CASE 2: Assume that $\w^2 - 1$ is a non-square in $\F_{q^{2n+1}}$. By Theorem \ref{fer}, assuming that $k=2n+1$, to find a non-zero root of $F_1$ in $\F_{q^{2n+1}}$ it is enough to show that $F_1$ has a non-zero root in $\F_{q^{4n+2}}$. Take $\lambda \in \F_q \setminus \{0\}$ as in \eqref{lambda}, so $-\lambda$ is a non-square in $\F_q$. By Corollary \ref{cor1}, assuming that $k=2n+1$, we may choose $\xi \in \F_{q^{4n+2}}$ such that $\xi^{q-1}= \w + T$. Denote by $L\in \F_{q^{4n+2}}$ any of the square-roots of $-\lambda$. We claim that 
$\xi / L$ is a root of $F_1$. 
Note that $(-\lambda)^{(q-1)/2} = -1$, 
i.e., $L^{q}=L (L^2)^{(q-1)/2} = -L$. Then
\begin{align*}
    F_1\left(\xi/L\right)L = & \xi^{q^2} - \left(\w^q + \w T^{q-1}\right) \xi^q + T^{q-1} \xi \\
	 = & \xi \left(\w + T\right)^{q+1}-\left(\w^q + \w T^{q-1}\right) \left(\w + T\right)\xi + T^{q-1} \xi.
\end{align*}
After dividing by $\xi$ we obtain
\begin{align*}
& \left( \w^q + T^q\right) \left(\w + T\right) - \left(\w^q + \w T^{q-1}\right) 
\left(\w + T\right) + T^{q-1} \\
& = \w^{q+1} + \w^q T + \w T^q + T^{q+1} - \w^{q+1} - \w^q T - \w^2 T^{q-1} - \w T^{q} + T^{q-1} \\	
& = T^{q+1} - \w^2 T^{q-1}+ T^{q-1}.
\end{align*}
Note that $T^2 - \w^2 = -1$, thus $T^{q+1} - \w^2 T^{q-1} = -T^{q-1}$. Therefore $F_1\left(\xi/L\right) = 0$.
\end{proof}

The following result holds true. 

\begin{theorem}
	\label{complete}
The cap $\cV_1 \cup \cV_\alpha$ of $\PG(4n+1,q)$, $q>2$, is a complete cap of size $2(q^{2n+1}-1)/(q-1)$.
\end{theorem}

\begin{cor}
Every point of $\PG(V) \setminus (\cV_1 \cup \cV_\alpha)$ lies on exactly one line joining a point of $\cV_1$ and a point of $\cV_\alpha$. 
\end{cor}
\begin{proof}
Note that there are $(q-1) \left(\frac{q^{2n+1}-1}{q-1}\right)^2$ points in $\PG(V) \setminus (\cV_1 \cup \cV_\alpha)$. By Theorem \ref{main}, every point of $\PG(V) \setminus (\cV_1 \cup \cV_\alpha)$ lies on at least a line joining a point of $\cV_1$ and a point of $\cV_{\alpha}$. On the other hand there are $\left(\frac{q^{2n+1}-1}{q-1}\right)^2$ lines meeting both $\cV_1$ and $\cV_\alpha$ in one point and on each of these lines there are $q-1$ points of $\PG(V) \setminus (\cV_1 \cup \cV_\alpha)$. The claim follows.
\end{proof}

\begin{remark}\label{oss_1}
Two consequences of the previous corollary arise. 

Firstly, the {\em join} of $\cV_1$ and $\cV_\alpha$ defined by
\[
\{(x,y,z) \in \cV_1 \times \cV_{\alpha} \times \PG(4n+1, q)  : \mbox{$z$ is a point of the line joining $x$ and $y$}\}.
\]
is the whole of $\PG(4n+1, q)$. In other words, the pointsets $\cV_1$ and $\cV_{\alpha}$ behave as two disjoint $2n$-dimensional projective subspaces of $\PG(4n+1, q)$.

Secondly, if $\ell_1$ and $\ell_\alpha$ are lines of $\PG(V)$ such that $\ell_1 \cap \cV_1 = \{V_1, V_2\}$ and $\ell_\alpha \cap \cV_\alpha = \{V_3, V_4\}$, then $|\ell_1 \cap \ell_\alpha| = 0$. Otherwise the lines $V_1 V_3$ and $V_2 V_4$ would have a point in common, contradicting the previous corollary.
\end{remark}

\section{A geometric description}

Let $\tilde{W}$ be the $\F_{q^{2n+1}}$--vector space of dimension $(n+1)(2n+1)$ consisting of the $(2n+1) \times (2n+1)$ symmetric matrices over $\F_{q^{2n+1}}$. Let $W$ be the subset of vectors $M(a_0, \dots, a_n)$, $a_0, \dots, a_n \in \F_{q^{2n+1}}$, of $\tilde{W}$, where $M(a_0, \dots a_n)=(m_{ij})$,$i,j\in\{1,\ldots,2n+1\}$, 
denotes the matrix 
$$
\begin{pmatrix}
a_0 & a_1 & a_2 & \dots & a_{n-1} & a_n & a_{n}^{q^{n+1}} & a_{n-1}^{q^{n+2}} & \dots & a_3^{q^{2n-2}} & a_{2}^{q^{2n-1}} & a_{1}^{q^{2n}} \\
a_1 & a_0^q & a_1^q & \dots & a_{n-2}^q & a_{n-1}^q & a_n^q & a_n^{q^{n+2}} & \dots & a_4^{q^{2n-2}} & a_3^{q^{2n-1}} & a_2^{q^{2n}} \\
a_2 & a_1^q & a_0^{q^2} & \dots & a_{n-3}^{q^2} & a_{n-2}^{q^2} & a_{n-1}^{q^2} & a_{n}^{q^2} & \dots & a_{5}^{q^{2n-2}} & a_{4}^{q^{2n-1}} & a_{3}^{q^{2n}} \\
\vdots & \vdots & \vdots & \ddots & \vdots & \vdots & \vdots & \vdots & \ddots & \vdots & \vdots & \vdots \\
a_{n-1} & a_{n-2}^q & a_{n-3}^{q^2} & \dots & a_{0}^{q^{n-1}} & a_{1}^{q^{n-1}} & a_2^{q^{n-1}} & a_3^{q^{n-1}} & \dots & a_{n-1}^{q^{n-1}} & a_{n}^{q^{n-1}} & a_n^{q^{2n}} \\
a_n & a_{n-1}^q & a_{n-2}^{q^2} & \dots & a_{1}^{q^{n-1}} & a_0^{q^n} & a_1^{q^n} & a_2^{q^n} & \dots & a_{n-2}^{q^n} & a_{n-1}^{q^{q^n}} & a_n^{q^n} \\
a_n^{q^{n+1}} & a_n^{q} & a_{n-1}^{q^2} & \dots & a_{2}^{q^{n-1}} & a_1^{q^n} & a_0^{q^{n+1}} & a_1^{q^{n+1}} & \dots & a_{n-3}^{q^{n+1}} & a_{n-2}^{q^{n+1}} & a_{n-1}^{q^{n+1}} \\
a_{n-1}^{q^{n+2}} & a_n^{q^{n+2}} & a_{n}^{q^2} & \dots & a_3^{q^{n-1}} & a_2^{q^{n}} & a_1^{q^{n+1}} & a_0^{q^{n+2}} & \dots & a_{n-4}^{q^{n+2}} & a_{n-3}^{q^{n+2}} & a_{n-2}^{q^{n+2}} \\ 
\vdots & \vdots & \vdots & \ddots & \vdots & \vdots & \vdots & \vdots & \ddots & \vdots & \vdots & \vdots \\
a_3^{q^{2n-2}} & a_{4}^{q^{2n-2}} & a_5^{q^{2n-2}} & \dots & a_{n-1}^{q^{n-1}} & a_{n-2}^{q^{n}} & a_{n-3}^{q^{n+1}} & a_{n-4}^{q^{n+2}} & \dots & a_0^{q^{2n-2}} & a_1^{q^{2n-2}} & a_2^{q^{2n-2}} \\
a_2^{q^{2n-1}} & a_{3}^{q^{2n-1}} & a_{4}^{q^{2n-1}} & \dots & a_n^{q^{n-1}} & a_{n-1}^{q^n} & a_{n-2}^{q^{n+1}} & a_{n-3}^{q^{n+2}} & \dots & a_1^{q^{2n-2}} & a_0^{q^{2n-1}} & a_1^{q^{2n-1}} \\
a_1^{q^{2n}} & a_{2}^{q^{2n}} & a_3^{q^{2n}} & \dots & a_{n}^{q^{2n}} & a_n^{q^n} & a_{n-1}^{q^{n+1}} & a_{n-2}^{q^{n+2}} & \dots & a_2^{q^{2n-2}} & a_1^{q^{2n-1}} & a_0^{q^{2n}} \\  
\end{pmatrix}.
$$
It is easily seen that $W$ is an $\F_q$--vector space of dimension $(n+1)(2n+1)$. Denote by $\tilde{\Pi}_i$ the $2n$-dimensional subspaces of $\PG(W)$ consisting of the points $M(0, \dots, 0, a_i, 0, \dots, 0)$, $a_i \in \F_{q^{2n+1}} \setminus \{0\}$. 
Let $\eta$ be a primitive element of $\F_{q^{2n+1}}$. Consider the invertible linear maps of $W$ given by
\begin{align}
& M(a_0, a_1, \dots, a_n) \in W \longmapsto M(a_0, \alpha_1 a_1, \dots, \alpha_n a_n) \in W \label{psi1} \\
& M(a_0, a_1, \dots, a_n) \in W \longmapsto M(\eta^2 a_0, \eta^{q+1} a_1, \dots, \eta^{q^{n}+1}a_n) \in W \label{phi1}
\end{align}
and denote by $\psi_{\alpha_1, \dots, \alpha_n}$ and $\tilde{\phi}$ the projectivities of $\PG(W)$ defined by the linear maps \eqref{psi1} and  \eqref{phi1}, respectively. The projectivity $\psi_{\alpha_1, \dots, \alpha_n}$ stabilizes setwise the subspaces $\tilde{\Pi}_i$, $0 \le i \le n$.

For $\alpha_i \in \F_{q^{2n+1}} \setminus \{0\}$, $1 \le i \le n$, let $\cV_{\alpha_1, \dots, \alpha_n}$ be the set of points of $\PG(W) \simeq \PG(n(2n+3), q)$ of the form
\[
M\left(x^2, \alpha_1 x^{q+1}, \dots, \alpha_n x^{q^{n}+1}\right), x \in \F_{q^{2n+1}} \setminus \{0\}.
\]
A point of $\PG(W)$ is said to have rank $r$ if and only if the corresponding matrix $M(a_0, \dots, a_n)$ has rank $r$. The projective variety consisting of the zeros of all determinants of the submatrices of order $2$ is known as the {\em standard Veronese variety} of $\PG(W)$, see \cite[Section 4.1]{HT}. By definition $\cV_{1, \dots, 1}$ is the standard Veronese variety of $\PG(W)$. Hence $|\cV_{1, \dots, 1}| = \frac{q^{2n+1}-1}{q-1}$. Similarly, the {\em chordal variety} $\cS_{1, \dots, 1}$ of the standard Veronese variety of $\PG(W)$ is formed by the zeros of all determinants of the submatrices of order $3$, see \cite[Lecture 9]{Harris}. 

\begin{remark}\label{oss_2}
Observe that every point of $\PG(W)$ lying on a line joining two points of $\cV_{1, \dots, 1}$ is contained in $\cS_{1, \dots, 1}$. On the other hand, when $q$ is odd, every point of $\cS_{1,\ldots,1}$ is contained in a line joining two points of $\cV_{1,\ldots,1}$, see \cite[Proposition 2.16]{CLS}.
\end{remark}

\begin{lemma}\label{lemma11}
\begin{itemize}
\item[1)] $\psi_{\alpha_1, \dots, \alpha_n}$ maps $\cV_{1, \dots, 1}$ to $\cV_{\alpha_1, \dots, \alpha_n}$.
\item[2)] $\langle \tilde{\phi} \rangle$ is a group of order $\frac{q^{2n+1}-1}{q-1}$ acting regularly on points of $\cV_{\alpha_1, \dots, \alpha_n}$.
\end{itemize}
\end{lemma}
\begin{proof}
Part {\em 1)} is straightforward.

To prove part {\em 2)}, consider the projectivity $\tilde{\phi}^i$ associated with the linear map given by $M(a_0, a_1, \dots, a_n) \in V \longmapsto M(\eta^{2i} a_0, \eta^{(q+1)i} a_1, \dots, \eta^{(q^{n}+1)i}a_n) \in V$. Therefore $\tilde{\phi}^i$ is the identity if and only if $\eta^{(q-1)i} = 1$ and hence $\eta^i \in \F_q$. Moreover $|Stab_{\langle \tilde{\phi} \rangle}(M(1, \dots, 1))| = 1$. Since $|\langle \tilde{\phi} \rangle| = |\cV_{\alpha_1, \dots, \alpha_n}|$, the result follows.
\end{proof}

\begin{lemma}
The $\left(q^{2n+1}-1\right)^n$ Veronese varieties $\cV_{\alpha_1, \dots, \alpha_n}$ form a partition of the points of $\PG(W)$ contained in none of the $(2n^2+n-1)$-dimensional subspaces $\langle \Pi_0, \dots, \Pi_{i-1}, \Pi_{i+1}, \dots, \Pi_n \rangle$, $0 \le i \le n$.
\end{lemma}
\begin{proof}
It is enough to show that if $(\alpha_1, \dots, \alpha_n) \ne (1, \dots, 1)$, then $M(1, \dots, 1) \notin \cV_{\alpha_1, \dots, \alpha_n}$. Indeed, the points of $\PG(W)$ contained in none of the  $(2n^2+n-1)$-dimensional subspaces $\langle \Pi_0, \dots, \Pi_{i-1}, \Pi_{i+1}, \dots, \Pi_n \rangle$, $0 \le i \le n$ are $\frac{\left(q^{2n+1}-1\right)}{q-1}^{n+1}$ in numbers. Moreover, if $\cV_{\alpha_1', \dots, \alpha_n'} = \cV_{\beta_1, \dots, \beta_n}$ with $(\alpha_1', \dots, \alpha_n') \ne (\beta_1, \dots, \beta_n)$, then by applying $\psi_{\alpha_1', \dots, \alpha_n'}^{-1}$, we find $\cV_{\alpha_1, \dots, \alpha_n} = \cV_{1, \dots, 1}$, with $(\alpha_1, \dots, \alpha_n) \ne (1, \dots, 1)$. Since $\langle \phi \rangle$ acts regularly on $\cV_{\alpha_1, \dots, \alpha_n}$, then $\cV_{\alpha_1, \dots, \alpha_n} = \cV_{1, \dots, 1}$ if and only if there is a point of $\cV_{1,\ldots,1}$ lying on $\cV_{\alpha_1, \dots, \alpha_n}$.

Suppose by contradiction that $M(1, \dots, 1)$ is a point of $\cV_{\alpha_1, \dots, \alpha_n}$ and $(\alpha_1, \dots, \alpha_n) \ne (1, \dots, 1)$. Then there exists $x \in \F_{q^{2n+1}} \setminus \{0\}$ such that $M(1,\dots,1) = M \left(x^2, \alpha_1 x^{q+1}, \alpha_2 x^{q^2+1}, \dots, \alpha_n x^{q^{n}+1} \right)$. By comparing the $2n$ elements $m_{1,i}$ and $m_{i,i}$, with $2 \le i \le n+1$, we get $x^{q^i-1} = \frac{1}{\alpha_i}, x^{2(q^i-1)} = 1$, for $1 \le i \le n$. This implies that $\N_{q^{2n+1}/q}(\alpha_i)=1$ and $\alpha_i^2=1$, i.e., $\alpha_i=1$, for $1 \le i \le n$, a contradiction.
\end{proof}

Let $\tilde{V}$ be the $(4n+2)$--dimensional $\F_q$--vector space underlying the projective space spanned by $\tilde{\Pi}_0, \tilde{\Pi}_1$ and let $\tilde{\cV}_{\w}$ be the set of points of $\PG(\tilde{V})$ obtained by projecting $\cV_{\w, \alpha_2, \dots, \alpha_n}$ from $\langle \tilde{\Pi}_2, \dots, \tilde{\Pi}_{n} \rangle$ onto $\PG(\tilde{V})$ when $n \geq 2$, and simply put $\tilde{\cV}_{\w}=\cV_{\w}$ when $n=1$. With the notation used in Section \ref{construction}, it is clear that $\tilde{V} \simeq V$ as $\F_q$-vector spaces. Moreover, the following result holds true.


\begin{prop}
	\label{projection}
	$\tilde{\cV}_{\w}$ and $\cV_{\w}$ are projectively equivalent.
\end{prop}
\begin{proof}
The $\F_q$-linear map
	\[
	M\left(x^2, \w x^{q+1}, \alpha_2 x^{q^2+1}, \dots, \alpha_n x^{q^{n}+1}\right) \mapsto v(x^2, \w x^{q+1})
	\] 
	is non-singular and induces a bijection between the points of $\tilde{\cV}_{\w}$ and the points of $\cV_{\w}$. 
\end{proof}

Let $\cS_{\alpha_1, \dots, \alpha_n}$ be the chordal variety of $\cV_{\alpha_1, \dots, \alpha_n}$, that is, $\cS_{\alpha_1, \dots, \alpha_n}$ is the image of $\cS_{1, \dots, 1}$ under ${\psi_{\alpha_1, \dots, \alpha_n}}$. 

\begin{remark}\label{chords0}
Since $\cS_{1, \dots, 1}$ is the set of zeros of the determinants of all $3 \times 3$ submatrices of $M(a_0, a_1, \dots, a_n)$ and $\cS_{\alpha, 1, \dots, 1}$ is obtained by substituting $a_1$ with $\alpha a_1$, it turns out that $\cS_{\alpha, 1, \dots, 1}$ is the set of zeros of the determinants of all $3 \times 3$ submatrices of $M(a_0, \alpha^{-1} a_1, \dots, a_n)$.
\end{remark}

From Remark \ref{oss_1} and Remark \ref{oss_2}, the following result can be deduced.

\begin{cor}
Let $q$ be odd. Then $\cS_{1, \alpha_2 \dots, \alpha_n}$ and $\cS_{-1, \beta_2, \dots, \beta_n}$ are disjoint.
\end{cor}

\underline{From now on we will assume $n = 1$.} 

In this case the same geometric setting has been used also in \cite{MarinoPepe,subspacecode} for $q$ odd. Note that if $q$ is even then the chordal variety of a Veronese variety $\cV$ contains a plane disjoint from $\cV$, called the {\em nuclear plane}, see for instance \cite{HT}. If $q$ is odd, in the partition above there are two Veronese varieties whose chordal varieties are disjoint. A further proof of this fact will be provided below. If $q$ is even we show that there are two Veronese varieties having the same nuclear plane such that their chordal varieties intersect exactly in the points of this plane. 

\begin{theorem}\label{chords}
Let $\alpha$ be $-1$ if $q$ is odd or an element of $\F_q \setminus \{0, 1\}$ if $q > 2$ is even. Then $\cS_{1}$ and $\cS_{\alpha}$ are disjoint if $q$ is odd or they meet only in the points of the common nuclear plane of $\cV_1$ and $\cV_{\alpha}$ if $q$ is even.
\end{theorem}
\begin{proof}

Assume by contradiction that $\cS_{1} \cap \cS_{\alpha}$ is not empty. Thus, by Remark \ref{chords0}, there is a matrix $M(a_0, a_1)$ such that $\rk M(a_0, a_1) \le 2$ and $\rk M(a_0, \alpha^{-1} a_1) \le 2$. Hence 
\begin{align*}
    & \N_{q^3/q} (a_0) + 2 \N_{q^3/q} (a_1) - a_0 a_1^{2q} - a_0^q a_1^{2q^2} - a_0^{q^2} a_1^{2} = 0, \\
    & \N_{q^3/q} (a_0) + 2 \alpha^{-3} \N_{q^3/q} (a_1) - \alpha^{-2} \left(a_0 a_1^{2q} + a_0^q a_1^{2q^2} + a_0^{q^2} a_1^{2}\right) = 0.
\end{align*}
Therefore we get $\N_{q^3/q}(a_1) = \N_{q^3/q}(a_0) = 0$ if $q$ is odd, whereas $\N_{q^3/q}(a_0)=0$ if $q$ is even. If $q$ is odd, the two varieties $\cS_1$ and $\cS_{-1}$ are disjoint. If $q$ is even the two varieties $\cS_1$ and $\cS_{\alpha}$ meet in points of the form $M(0,a_1)$, which belong to the nuclear plane of $\cV_1$ and $\cV_{\alpha}$, see \cite[Theorem 4.23]{HT} and \cite[Section 2.3]{MarinoPepe}.
\end{proof}

By Theorem \ref{chords} and Proposition \ref{projection},  $\cV_1$ and $\cV_{\alpha}$ (where $\alpha$ is $-1$ if $q$ is odd or an element $\F_q \setminus \{0, 1\}$ if $q > 2$ is even) are Veronese varieties of $\PG(V) \simeq \PG(5, q)$ such that their chordal varieties are disjoint or they meet only in the points of the common nuclear plane of $\cV_1$ and $\cV_{\alpha}$. This is enough to provide alternative proofs for Proposition \ref{cap1} and Theorem \ref{complete}.    

\begin{theorem}
Let $\cV$, $\cV'$ be two Veronese varieties of $\PG(5, q)$ such that their chordal varieties are disjoint if $q$ is odd or they meet only in the points of the common nuclear plane of $\cV_1$ and $\cV_{\alpha}$ if $q$ is even. Then $\cV \cup \cV'$ is a complete cap of $\PG(5, q)$.
\end{theorem}
\begin{proof}
By the hypothesis $\cV$ and $\cV'$ are disjoint. 

Let $q$ be odd. We claim that $\cV \cup \cV'$ is a cap. Assume by contradiction and w.l.o.g. that there is a line sharing two points with $\cV$ and one point with $\cV'$. Then the two chordal varieties have a point in common, which is not the case. We show that $\cV \cup \cV'$ is complete. Let $\ell_1, \ell_2$ be two lines of $\PG(5, q)$ such that $\ell_1 = P_1 Q_1$ and $\ell_2 = P_2 Q_2$, where $P_1, P_2$ are points of $\cV$ and $Q_1, Q_2$ are points of $\cV'$. We claim that $|\ell_1 \cap \ell_2| = 0$. Otherwise if $P = \ell_1 \cap \ell_2$, then $\langle \ell_1, \ell_2 \rangle$ is a plane, say $\pi$, where $P_1, P_2, Q_1, Q_2 \in \pi$. Hence $P_1 P_2 \cap Q_1 Q_2$ is a common point of the two chordal varieties, a contradiction. There are $(q^2+q+1)^2$ lines having exactly one point with both $\cV$ and $\cV'$ and by the above argument every point of $\PG(5, q) \setminus (\cV \cup \cV')$ lies on at most one of these lines. It follows that there are $(q^3-1)(q^2+q+1)$ points of $\PG(5, q)\setminus (\cV \cup \cV')$ covered by these lines where $|\PG(5, q) \setminus (\cV \cup \cV')| = (q^3-1)(q^2+q+1)$. This completes the proof.

Let $q$ be even. From the proof of \cite[Theorem 4.23]{HT} a line having two points in common with a Veronese variety is disjoint from its nuclear plane. By repeating the same argument as in the odd characteristic case the result follows. 
\end{proof}

\begin{remark}
{\rm Assume $q$ is odd. Under the Veronese bijective correspondence between conics of $\PG(2,q)$ and points of $\PG(5,q)$, let us consider $\cV$ as the Veronese surface arising from the lines of $\PG(2, q)$ counted twice and $\cV'$ as the Veronese surface obtained from the conics of an inscribed projective bundle, i.e. a particular collection of non--degenerate conics of $\PG(2,q)$ that mutually intersect in exactly one point, see \cite{BBEF}. By using the classification of pencils of conics of $\PG(2, q)$, \cite[Table 7.7]{H1}, \cite{LP}, it is possible to show that $\cV, \cV'$ satisfy the hypothesis of the previous theorem.} 
\end{remark}

\bigskip

\smallskip
{\footnotesize
\noindent\textit{Acknowledgments.}
This work was supported by the Italian National Group for Algebraic and Geometric Structures and their Applications (GNSAGA-- INdAM).
The second author was supported by the National Research, Development and Innovation Office -- NKFIH, grants no. PD 132463 and K 124950.}

\end{document}